\documentclass[12pt,a4paper,leqno,twoside]{amsart}
\usepackage[a4paper,left=3cm,top=4cm,bottom=4cm,right=3cm]{geometry}
\usepackage[T1]{fontenc}
\usepackage{amssymb,bbm}
\usepackage{amsmath,amsthm,dsfont}
\usepackage{amsfonts,mathrsfs,wasysym}
\usepackage{latexsym}
\usepackage[cp1250]{inputenc}
\usepackage{graphicx}
\usepackage{indentfirst}
\usepackage{enumerate}
\usepackage{geometry}
\usepackage{enumitem}
\usepackage{xparse}
\usepackage{etoolbox}
\usepackage{mathptmx}
\usepackage{fancyhdr}
\usepackage[usenames,dvipsnames]{xcolor}
\usepackage{marvosym}
\usepackage{pifont}
\usepackage{cite}

\usepackage{hyperref}
\hypersetup{colorlinks,linkcolor={blue},citecolor={red}}

\patchcmd{\abstract}{\scshape\abstractname}{\textbf{\abstractname}}{}{}

\makeatletter
\DeclareMathAlphabet{\mathcal}{OMS}{cmsy}{m}{n}

\DeclareSymbolFont{operators}{OT1}{ztmcm}{m}{n}
\DeclareSymbolFont{letters}{OML}{ztmcm}{m}{it}
\DeclareSymbolFont{symbols}{OMS}{ztmcm}{m}{n}
\DeclareSymbolFont{largesymbols}{OMX}{ztmcm}{m}{n}
\DeclareSymbolFont{bold}{OT1}{ptm}{bx}{n}
\DeclareSymbolFont{italic}{OT1}{ptm}{m}{it}
\@ifundefined{mathbf}{}{\DeclareMathAlphabet{\mathbf}{OT1}{ptm}{bx}{n}}
\@ifundefined{mathit}{}{\DeclareMathAlphabet{\mathit}{OT1}{ptm}{m}{it}}
\DeclareMathSymbol{\omicron}{0}{operators}{`\o}

\DeclareMathAlphabet{\mathpzc}{OT1}{pzc}{m}{it}
\DeclareSymbolFont{operators}{OT1}{txr}{m}{n}
\SetSymbolFont{operators}{bold}{OT1}{txr}{bx}{n}
\def\operator@font{\mathgroup\symoperators}
\DeclareSymbolFont{italic}{OT1}{txr}{m}{it}
\SetSymbolFont{italic}{bold}{OT1}{txr}{bx}{it}
\DeclareSymbolFontAlphabet{\mathrm}{operators}
\DeclareMathAlphabet{\mathbf}{OT1}{txr}{bx}{n}
\DeclareMathAlphabet{\mathit}{OT1}{txr}{m}{it}
\SetMathAlphabet{\mathit}{bold}{OT1}{txr}{bx}{it}
\DeclareSymbolFont{letters}{OML}{txmi}{m}{it}
\SetSymbolFont{letters}{bold}{OML}{txmi}{bx}{it}
\DeclareFontSubstitution{OML}{txmi}{m}{it}
\DeclareSymbolFont{lettersA}{U}{txmia}{m}{it}
\SetSymbolFont{lettersA}{bold}{U}{txmia}{bx}{it}
\DeclareFontSubstitution{U}{txmia}{m}{it}
\DeclareSymbolFontAlphabet{\mathfrak}{lettersA}
\DeclareSymbolFont{symbols}{OMS}{txsy}{m}{n}
\SetSymbolFont{symbols}{bold}{OMS}{txsy}{bx}{n}
\DeclareFontSubstitution{OMS}{txsy}{m}{n}
\makeatother

\makeatletter
\renewcommand\abstractname{\scshape\bfseries Abstract}

\renewenvironment{proof}[1][\proofname]{\par \pushQED{\qed} \normalfont
  \topsep6\p@\@plus6\p@ \trivlist \itemindent\z@
  \item[\hskip\labelsep\bfseries
    #1\@addpunct{.}]\ignorespaces
}{
  \popQED\endtrivlist\@endpefalse
}

\makeatother

\makeatletter
    \@addtoreset{equation}{section} 
    \renewcommand{\theequation}{{\thesection}.\@arabic\c@equation} 

\makeatother

\makeatletter

\def\section{\@ifstar\unnumberedsection\numberedsection}
\def\numberedsection{\@ifnextchar[
  \numberedsectionwithtwoarguments\numberedsectionwithoneargument}
\def\unnumberedsection{\@ifnextchar[
  \unnumberedsectionwithtwoarguments\unnumberedsectionwithoneargument}
\def\numberedsectionwithoneargument#1{\numberedsectionwithtwoarguments[#1]{#1}}
\def\unnumberedsectionwithoneargument#1{\unnumberedsectionwithtwoarguments[#1]{#1}}
\def\numberedsectionwithtwoarguments[#1]#2{%
  \ifhmode\par\fi
  \removelastskip
  \vskip 4ex\goodbreak
  \refstepcounter{section}%
  \noindent
  \begingroup
  \leavevmode\centering\scshape\bfseries
  \thesection.
  #2
  \par
  \endgroup
  \vskip 1ex\nobreak
  \addcontentsline{toc}{section}{%
    \protect\numberline{\thesection}%
    #1}%
  }
\def\unnumberedsectionwithtwoarguments[#1]#2{%
  \ifhmode\par\fi
  \removelastskip
  \vskip 2ex\goodbreak
  \noindent
  \begingroup
  \leavevmode\centering\scshape\bfseries
  \leavevmode\centering\scshape\bfseries
  #2
  \par
  \endgroup
  \vskip 1ex\nobreak
  \addcontentsline{toc}{section}{%
    #1}%
}
\makeatother

\makeatletter
\def\@seccntformat#1{\csname mythe#1\endcsname}

\let\latex@subsection\subsection
\def\subsection{\@ifstar{\refstepcounter{subsection}\latex@subsection*}{\latex@subsection}}
\def\@makechapterhead#1{%
  \vspace*{40\p@}%
  {\parindent \z@ \raggedright \normalfont
    \interlinepenalty\@M
    \Huge \bfseries #1\par \nobreak
    \vskip 40\p@
  }}
\let\latex@l@chapter\l@chapter
\def\l@chapter#1#2{\begingroup\let\numberline\@gobble\latex@l@chapter{#1}{#2}\endgroup}
\makeatother

\theoremstyle{plain}
\newtheorem{Th}{Theorem}[section]
\newtheorem{Prop}[Th]{Proposition}

\newtheorem{Cor}[Th]{Corollary}
\theoremstyle{definition}
\newtheorem{Rem}[Th]{Remark}
\newtheorem{Ex}[Th]{Example}
\newtheorem{Def}[Th]{Definition}

\def\bf{\textbf}
\def\it{\textit}
\def\tn{\textnormal}

\def\leq{\leqslant}
\def\geq{\geqslant}
\def\R{{\mathds R}}
\def\N{{\mathds N}}

\begin{document}
\vspace*{-4mm}
\title{An initial condition  reconstruction  in Hamilton-Jacobi equations}

\author{\vspace*{-0.2cm}{Arkadiusz Misztela$^{\tn{\textdagger}}$ and S\l{}awomir Plaskacz $^{\tn{\textdaggerdbl}}$}\vspace*{-0.2cm}}
\thanks{\textdagger\, Institute of Mathematics, University of Szczecin, Wielkopolska 15, 70-451 Szczecin, Poland; e-mail: arkadiusz.misztela@usz.edu.pl}
\thanks{\textdaggerdbl\, Nicolaus Copernicus University, Gagarina 11, 87-100 Toruń, Poland; e-mail: plaskacz@mat.umk.pl}

\begin{abstract}
We describe the family of initial conditions for Hamilton-Jacobi equations (HJE) corresponding to optimal control problems that can be retrieved by solving HJE backward in time.
 \\ \vspace{0mm}

\hspace{-1cm}
\noindent \bf{\scshape Keywords.} Hamilton-Jacobi equations, forward and backward viscosity solutions, value \\\hspace*{-0.55cm} function, Bolza problem, inverse Bolza problem.


\vspace{3mm}\hspace{-1cm}
\noindent \bf{\scshape Mathematics Subject Classification.} 49N45, 35F21, 49L25, 35Q93, 34A60.
\end{abstract}

\maketitle

\pagestyle{myheadings}  \markboth{\small{\scshape Arkadiusz Misztela and Sławomir Plaskacz}
}{\small{\scshape Reconstruction of Initial Condition}}

\thispagestyle{empty}

\vspace{-0.8cm}

\section{Introduction}\label{section1}

\noindent Suppose that a function $u:[0,T]\times\R^n\to\R$ is a solution of the Hamilton-Jacobi equation
\begin{equation}\label{HJ1}
\triangledown_{\!t}\,u+H(x,\triangledown_{\!x}\,u)=0 \;\;\;\tn{in}\;\;\; (0,T)\times\R^n
\end{equation}
and it satisfies an initial  condition $u(0,\cdot)=g_0(\cdot)$, where $g_0:\R^n\to \R$ is a given function. We define a terminal function $g_T(\cdot):=u(T,\cdot)$. We say that the initial function $g_0$ can be  reconstructed if knowing the function $g_T$ and the Hamiltonian $H$ in \eqref{HJ1} we can obtain (reconstruct) the function $g_0$. If $u$ is a classical solution of \eqref{HJ1} with the initial condition $u(0,\cdot)=g_0(\cdot)$, then the function $w:=u$ is a (unique) classical solution of \eqref{HJ1} satisfying the terminal condition $w(T,\cdot)=g_T(\cdot)$ and $g_0$ is reconstructible as $g_0(\cdot)=w(0,\cdot)$. If $u$ is a  viscosity solution of \eqref{HJ1}, then this scheme usually  does not work, as we can see in the following example. The example is preceded by the definition of forward and backward viscosity solutions of Hamilton-Jacobi equation.

\begin{Def}\label{vs}
We say that a continuous function $u:[0,T]\times\R^n\to\R$ is a forward viscosity solution of \eqref{HJ1} if for every $(t,x)\in(0,T)\times\R^n$ we have
\begin{align}
& u_t+H(x,u_x)\leq 0\;\;\;\tn{for all}\;\;\;(u_t,u_x)\in\partial_{+}u(t,x),\label{fsub}\\
& u_t+H(x,u_x)\geq 0\;\;\;\tn{for all}\;\;\;(u_t,u_x)\in\partial_{-}u(t,x),\label{fsuper}
\end{align}
where $(u_t,u_x)\in\partial_{+}u(t,x)$ if there exists $\varphi\in C^1((0,T)\times\R^n,\R)$ such that $u-\varphi$ has a\linebreak local maximum at $(t,x)$ with $(u_t,u_x)=\triangledown\varphi(t,x)$
and $(u_t,u_x)\in \partial_{-}u(t,x)$ if there exists\linebreak $\varphi\in C^1((0,T)\times\R^n,\R)$ such that $u-\varphi$ has a local minimum at $(t,x)$ with $(u_t,u_x)=\triangledown\varphi(t,x)$.

We say that a continuous function $w:[0,T]\times\R^n\to\R$ is  a backward viscosity solution of \eqref{HJ1} if for every $(t,x)\in(0,T)\times\R^n$ we have
\begin{align}
& w_t+H(x,w_x)\geq 0\;\;\;\tn{for all}\;\;\;(w_t,w_x)\in\partial_{+}w(t,x),\label{bsub}\\
& w_t+H(x,w_x)\leq 0\;\;\;\tn{for all}\;\;\;(w_t,w_x)\in\partial_{-}w(t,x).\label{bsuper}
\end{align}
\end{Def}

\begin{Rem}
The terms  `backward viscosity solution' and  `forward viscosity solution' have not been used earlier (according to our knowledge).
We used them in the paper to highlight their relations to the (forward) Bolza problem and the inverse (backward)\linebreak Bolza problem. Obviously, what we call `forward viscosity solution' is nothing else as a `viscosity solution' (comp.\cite{B-CD}). Moreover, a function $w(t,x)$ is a backward viscosity\linebreak solution of  $\triangledown_{\!t}\,w+H(x,\triangledown_{\!x}\,w)=0$ if and only if the function $u(t,x)=-w(T-t,x)$ is \linebreak a  viscosity solution of  $\triangledown_{\!t}\,u+H(x,-\triangledown_{\!x}\,u)=0$.
\end{Rem}

\begin{Ex}
We consider the Hamilton-Jacobi equation
\begin{equation}\label{HJ2}
\triangledown_{\!t}\,u+|\triangledown_{\!x}\,u|=0 \;\;\;\tn{in}\;\;\; (0,T)\times\R
\end{equation}
with the initial function $g_0(x)=\max\{\,0,\,T-|x|\,\}$. The unique forward viscosity solution of~\eqref{HJ2} satisfying the initial condition $u(0,x)=g_0(x)$  is given by
$$u(t,x)=\max\{\,0,\,T-t-|x|\,\}.$$
Hence $g_T(x)=u(T,x)=0$. The function $w\equiv 0$ is the unique solution (classical, so any) of \eqref{HJ2} satisfying the terminal condition $w(T,\cdot)=g_T(\cdot)$. Therefore the initial function $g_0$ cannot be reconstructed.
\end{Ex}

 The question arises whether the initial condition can be reconstructed in any other case except the case when the solution $u$ of \eqref{HJ1} is classical. The following examples illustrate that it is possible.

\begin{Ex}
We consider the Hamilton-Jacobi equation \eqref{HJ2} with the initial function $g_0(x)=|x|$. The unique forward viscosity solution of \eqref{HJ2} satisfying the initial condition  $u(0,x)=g_0(x)$ is given by
$$u(t,x)=\max\{\,0,\,|x|-t\,\}.$$
So  $g_T(x)=u(T,x)=\max\{0,|x|-T\}$. The function $w:=u$ is the unique backward viscosity solution of \eqref{HJ2} satisfying the terminal condition  $w(T,\cdot)=g_T(\cdot)$. Therefore the initial function $g_0$ can be reconstructed.
\end{Ex}

\begin{Def}\label{Reconstruction}
We say that an initial function $g_0\in C(\R^n,\,\R)$ is reconstructible in time $T$ ($T>0$) if $g_0(\cdot)=w(0,\cdot)$, where $w:[0,T]\times \R^n\to\R$ is a unique backward viscosity solution of \eqref{HJ1} satisfying $w(T,\cdot)=u(T,\cdot)$ and $u:[0,T]\times \R^n\to\R$ is a unique forward viscosity solution of \eqref{HJ1} satisfying $u(0,\cdot)=g_0(\cdot)$.
\end{Def}

Note that if a forward viscosity solution $u$ of \eqref{HJ1}  satisfying  $u(0,\cdot)=g_0(\cdot)$  is a backward  viscosity solution of \eqref{HJ1}, then an initial function $g_0$ is reconstructible. In the monograph\linebreak of Bardi and Capuzzo-Dolcetta \cite{B-CD} a solution of \eqref{HJ1} that is forward and backward\linebreak viscosity solution is called a bilateral solution. In \cite{C-F} Cannarsa-Frankowska showed that the value function is a bilateral solution at points being interior points of optimal trajectories. A global result saying that a forward viscosity solution $u$ is a bilateral one holds if the function $u$ is semiconvex. A sufficient condition is given in Theorem 7.4.13 that is presented in the monograph of Cannarsa-Sinestrari \cite{C-S-2004}. Below we provide an example of a bilateral solution $u$ that is neither semiconvex nor semiconcave.

\begin{Ex}\label{ncnc}
The function   $u:[0,T]\times\R\to\R$ given by the formula
$$u(t,x)=\max\{\,1-|x|\operatorname{e}^{t-T},\,|x|\operatorname{e}^{T-t}-1,\,(\operatorname{e}^{2T}-1)/(\operatorname{e}^{2T}+1)\,\}$$ is a bilateral solution of \eqref{HJ1} with $H(x,p)=|x|\,|p|$. Therefore the initial function $g_0(x)=u(0,x)=\max\{1-|x|\operatorname{e}^{-T},|x|\operatorname{e}^T-1\}$ is reconstructible in time $T>0$ and $u=w$. Moreover, we notice that the initial  function $g_0$ may be obtained as the restriction $g_0(\cdot)=v(0,\cdot)$ of the backward solution $v:[0,T]\times\R\to\R$ given by
$$v(t,x)=\max\{\,1-|x|\operatorname{e}^{t-T},\,|x|\operatorname{e}^{T-t}-1\,\}.$$
\end{Ex}

Our main observation in the paper is that if $g_0=v(0,\cdot)$, where $v:[0,T]\times\R^n\to\R$ is a backward viscosity solution of \eqref{HJ1}, then the initial function $g_0$ is reconstructible in time $T$. This result can be formulated as follows.

\begin{Th}\label{main}
Assume that the Hamiltonian $H:\R^n\times\R^n\to \R$ satisfies
\begin{equation}\label{A}
\left\{\begin{array}{l}
p\to H(x,p)\;\; \it{is convex for every}\;\;  x\in\R^n,\\[0.3mm]
|H(x,p)-H(x,q)|\leq M(1+|x|)\,|p-q|\;\; \it{for all}\;\; x,p,q\in\R^n,\\[0.3mm]
|H(x,p)-H(y,p)|\leq M(1+|p|)\,|x-y|\;\; \it{for all}\;\; x,y,p\in\R^n,\\[0.3mm]
\it{and some constant}\;\; M\geq 0.
\end{array}\right.\tag{A}
\end{equation}
Then an initial function $g_0\in C(\R^n,\,\R)$ is reconstructible in time $T>0$ if and only if there exists a backward viscosity solution $v$ of \eqref{HJ1} defined on $[0,T]\times\R^n$ and such that $g_0(\cdot)=v(0,\cdot)$. Furthermore, if the function $u:[0,T]\times\R^n\to\R$ is a forward viscosity solution of \eqref{HJ1} satisfying $u(0,\cdot)=g_0(\cdot)$ and the function $w:[0,T]\times\R^n\!\to\R$ is a backward viscosity solution of \eqref{HJ1} satisfying $w(T,\cdot)=u(T,\cdot)$, then $u\geq w\geq v$.
\end{Th}

The following example shows that the backward viscosity solution  $v$ in Theorem \ref{main} satisfying $g_0(\cdot)=v(0,\cdot)$ is usually not unique.
\begin{Ex}
 The function   $v_{\alpha}:[0,T]\times\R\to\R$ given by the formula
$$v_\alpha(t,x)=\alpha\,\min\{\,0,\,|x|-t\,\}$$ is a backward viscosity solution of \eqref{HJ2} satisfying $v_{\alpha}(0,\cdot)\equiv 0$ for all $\alpha\geq 0$. The initial function $g_0\equiv 0$ is reconstructible in time $T>0$ with $u\equiv w \equiv 0$.
\end{Ex}

We provide an example of an initial function $g_0$ that is reconstructible in time $T\in(0,1]$ and is not reconstructible in any time $T>1$.
\begin{Ex}
The function $u:[0,\infty)\times\R\to\R$ given by the formula
$$u(t,x)=\min\{\,0,-|x|+1-t\,\}$$
is a forward viscosity solution of \eqref{HJ2}. We set $g_T(x):=u(T,x)$ with $T>0$. The unique backward viscosity solution $w:[0,T]\times\R\to\R$ of \eqref{HJ2} satisfying the terminal condition $w(T,\cdot)=g_T(\cdot)$ is given by the formula
$$w(t,x)=\min\{\,0,1-T,-|x|+1-t\,\}.$$
We observe that $u\equiv w$ for all $T\in(0,1]$. Thus, the initial function $g_0(x)=\min\{0,-|x|+1\}$ is reconstructible in time $T\in(0,1]$. However, it is not reconstructible in any time $T>1$ because
$w(0,0)=1-T\neq 0=g_0(0)$.
\end{Ex}

The  question that arises for reconstructible initial functions is whether the forward viscosity solution $u$ equals to the backward viscosity solution $w$ on the whole domain $[0,T]\times\R^n$.
We succeeded to prove that $u=w$ for reconstructible initial function only in the case where the dimension $n=1$ and for Hamiltonians $H$ corresponding to a Mayer control problem (see Subsection \ref{s-cmp}). However, in Subsection \ref{s-dmp} we provide  an example of a discrete time Mayer problem in a finite state space such that the value functions corresponding to $u$ and $w$ are not equal. It does not mean that we have an example of $H$  satisfying the assumptions of Theorem \ref{main} and   an reconstructible initial function $g_0$ for which $u\neq w$. In Section \ref{pmr}, we provide two proofs of Theorem \ref{main}. The first one  bases on some classical properties of viscosity solutions. In the second one we use properties of the  value functions in corresponding to \eqref{HJ1} optimal control problems.  In both methods of the proof we abstract a common scheme called \it{ General setting} (see Section \ref{s-gs}) that we apply in some similar problems. In Subsection \ref{s-cbp}, we consider concave-convex Hamiltonians and we show that a forward viscosity solution corresponding to a convex initial function is a bilateral solution. It follows, that convex initial functions are reconstructible in every time $T>0$.

\section{Some preliminaries on viscosity solutions and optimal control}\label{prel}

\noindent Following \cite{B-CD}, we recall that a continuous function $v:[0,T]\times\R^n\to\R$ is a bilateral\linebreak viscosity subsolution of \eqref{HJ1} if for every $(t,x)\in(0,T)\times\R^n$ we have
\begin{equation*}
v_t+H(x,v_x)= 0\;\;\;\;\tn{for all}\;\;\;\;(v_t,v_x)\in\partial_{+}v(t,x).
\end{equation*}
A continuous function $v:[0,\!T]\times\R^n\to\R$ is a bilateral viscosity supersolution of \eqref{HJ1} if for every $(t,x)\in(0,T)\times\R^n$ we have
\begin{equation*}
v_t+H(x,v_x)= 0\;\;\;\;\tn{for all}\;\;\;\;(v_t,v_x)\in\partial_{-}v(t,x).
\end{equation*}
A continuous function $v:[0,T]\times\R^n\to\R$ is a bilateral viscosity solution of \eqref{HJ1} if it is a bilateral sub- and supersolution of \eqref{HJ1}. Obviously, a continuous function $v$ is a bilateral viscosity solution if and only if $v$ is forward and backward viscosity solution, i.e. conditions \eqref{fsub}, \eqref{fsuper}, \eqref{bsub}, \eqref{bsuper} are satisfied.

\vspace{4mm}
We say that a continuous function $u:[0,T]\times\R^n\to\R$ is a forward viscosity subsolution  [respectively, supersolution] of \eqref{HJ1} if for every $(t,x)\in(0,T)\times\R^n$ the condition \eqref{fsub} [respectively, \eqref{fsuper}] holds. In particular, $u$ is a forward viscosity solution of \eqref{HJ1}  if and only if  $u$ is a forward viscosity sub- and supersolution of \eqref{HJ1}.

\vspace{4mm}
We say that a continuous function $w\!:\![0,T]\times\R^n\!\!\to\!\R$ is a backward viscosity subsolution [respectively, supersolution] of \eqref{HJ1} if for every $(t,x)\in(0,T)\times\R^n$ the condition \eqref{bsub} [respectively, \eqref{bsuper}] holds. In particular, $w$ is a backward viscosity solution of \eqref{HJ1} if and only if $w$ is a backward viscosity sub- and supersolution of \eqref{HJ1}.

\vspace{4mm}
By Theorem 3.15 in \cite[Chap. 3]{B-CD} and Theorem 5.6 in \cite[Chap. 2]{B-CD} we obtain

\pagebreak
\begin{Th}\label{OT1}
Assume that $H:\R^n\times\R^n\to\R$ satisfies \eqref{A}. Then
\begin{enumerate}[leftmargin=11mm]
\item[\tn{\bf{(a)}}] if continuous functions $\bar{u},\tilde{u}:[0,T]\times\R^n\to\R$ are, respectively, a forward viscosity sub- and supersolution of \eqref{HJ1} and  $\bar{u}(0,\cdot)\leq \tilde{u}(0,\cdot)$, then $\bar{u}\leq\tilde{u}$\tn{;}
\item[\tn{\bf{(b)}}] if continuous functions $\bar{w},\tilde{w}\!:\![0,T]\times\R^n\!\!\to\!\R$ are, respectively, a backward viscosity sub- and supersolution of \eqref{HJ1} and  $\bar{w}(T,\cdot)\leq \tilde{w}(T,\cdot)$, then $\bar{w}\leq\tilde{w}$\tn{;}
\item[\tn{\bf{(c)}}]  a continuous function $u:[0,T]\times\R^n\to\R$ is a forward viscosity solution of \eqref{HJ1} if and only if $u$ is a bilateral viscosity supersolution and a continuous function $w:[0,T]\times\R^n\to\R$ is a backward viscosity solution of \eqref{HJ1} if and only if $w$ is a bilateral viscosity subsolution.
\end{enumerate}
\end{Th}

Assume that the Hamiltonian $H:\R^n\times\R^n\to \R$ satisfies \eqref{A} and $g_0,\,g_T\in C(\R^n,\R)$. There are connections between solutions of \eqref{HJ1} and optimal control problems given by a function dual to the Hamiltonian $H$. This function $L$, called the Lagrangian, is obtained from $H$ using the Legenndre-Fenchel transform:
\begin{equation}\label{tran1}
 L(x,v)= \sup_{p\in\R^{n}}\,\{\,\langle v,p\rangle-H(x,p)\,\}.
\end{equation}

For a given initial function $g_0$ and an initial pair $(t_0,x_0)$ the Bolza problem consists in
\begin{align*}
\mathrm{minimize}&\;\;\;g_0(x(0))+\int_{0}^{t_0}L(x(s),\dot{x}(s))\,\it{ds},\\[-1mm]
\mathrm{subject\;\, to}&\;\;\;x(\cdot)\in \mathcal{A}([0,t_0],\R^n)\;\;\tn{and}\;\;x(t_0)=x_0,
\end{align*}
where $\mathcal{A}([a,b],\R^n)$ denotes the space of all absolutely continuous functions from $[a,b]$ into $\R^n$. The forward value function $U:[0,T]\times\R^n\to\R$ corresponding to the Bolza problem is defined by the formula
\begin{equation}\label{fvf1}
U(t_0,x_0)= \inf_{\begin{array}{c}
\scriptstyle x(\cdot)\,\in\,\mathcal{A}([0,t_0],\R^n)\\[-1mm]
\scriptstyle x(t_0)=x_0
\end{array}}\,\Big\{\,g_0(x(0))+\int_{0}^{t_0}L(x(s),\dot{x}(s))\,\it{ds}\,\Big\}.
\end{equation}
We say that  an absolutely continuous function  $\bar{x}:[0,t_0]\to\R^n$ is an optimal trajectory on the time interval $[0,t_0]$ in the Bolza problem if
$U(t_0,\bar{x}(t_0))=U(0,\bar{x}(0))+\int_{0}^{t_0} L(\bar{x}(s),\dot{\bar{x}}(s))\it{ds}$. If $\bar{x}:[0,t_0]\to\R^n$ is the optimal trajectory in the Bolza problem, then we get
$U(t_2,\bar{x}(_2))=U(t_1,\bar{x}(t_1))+\int_{t_1}^{t_2} L(\bar{x}(s),\dot{\bar{x}}(s))\,\it{ds}$ for all $0\leq t_1<t_2\leq t_0$.
We say that an absolutely continuous function $\bar{x}:[0,T]\to\R^n$ is a maximal optimal trajectory in the Bolza problem if it is optimal on the time-interval $[0,T]$.

\vspace{1mm}
 For a given terminal function $g_T$ and a pair $(t_0,x_0)$ the inverse Bolza problem consists~in
\begin{align*}
\mathrm{maximize}&\;\;\;g_T(x(T))-\int_{t_0}^{T}L(x(s),\dot{x}(s))\,\it{ds},\\[-1mm]
\mathrm{subject\;\, to}&\;\;\;x(\cdot)\in \mathcal{A}([t_0,T],\R^n)\;\;\tn{and}\;\;x(t_0)=x_0.
\end{align*}

\vspace{1mm}
\noindent The backward value function $W\!\!:\![0,\!T]\!\times\!\R^n\!\to\!\R$ corresponding to inverse Bolza problem~is

\begin{equation}\label{fvf2}
W(t_0,x_0)= \sup_{\begin{array}{c}
\scriptstyle x(\cdot)\,\in\,\mathcal{A}([t_0,T],\R^n)\\[-1mm]
\scriptstyle x(t_0)=x_0
\end{array}}\,\Big\{\,g_T(x(T))-\int_{t_0}^{T}L(x(s),\dot{x}(s))\,\it{ds}\,\Big\}.
\end{equation}

We say that an absolutely continuous function $\tilde{x}:[t_0,T]\to\R^n$ is an optimal trajectory on the time interval $[t_0,T]$ in the inverse Bolza problem if
$W(t_0,\tilde{x}(t_0))=W(T,\tilde{x}(T))-\int_{t_0}^{T} L(\tilde{x}(s),\dot{\tilde{x}}(s))\,\it{ds}$. If $\tilde{x}\!:\![t_0,T]\to\R^n$ is the optimal trajectory in the inverse Bolza problem, then
$W(t_1,\tilde{x}(t_1))=W(t_2,\tilde{x}(t_2))-\int_{t_1}^{t_2} L(\tilde{x}(s),\dot{\tilde{x}}(s))\,\it{ds}$ for all $t_0\leq t_1<t_2\leq T$.
We say that an absolutely continuous function $\tilde{x}:[0,T]\to\R^n$ is a maximal optimal trajectory in the inverse Bolza problem if
it is optimal on the time-interval $[0,T]$.
By Theorem \ref{OT1}  and results concerning regularities of value functions from \cite{AM2,AM,AM1,P-Q,P-Q-2} we obtain the following:
\begin{Th}\label{OT}
Assume that $H:\R^n\times\R^n\to\R$ satisfies \eqref{A}  and $g_0(\cdot),\,g_T(\cdot)\in C(\R^n,\R)$. Let  $L$ be given by  \eqref{tran1}. Then we have the following.
\begin{enumerate}[leftmargin=11mm]
\item[\tn{\bf{(a)}}] The forward value function $U$ is the unique forward viscosity solution of \eqref{HJ1} satisfying $U(0,\cdot)=g_0(\cdot)$ and the backward value function $W$ is the unique backward viscosity solution of \eqref{HJ1} satisfying $W(T,\cdot)=g_T(\cdot)$\tn{;}
\item[\tn{\bf{(b)}}] For every $(t_0,x_0)\in[0,T]\times\R^n$ there exists an optimal trajectory $\bar{x}:[0,t_0]\to\R^n$ in the Bolza problem satisfying $\bar{x}(t_0)=x_0$\tn{;}
\item[\tn{\bf{(c)}}] For every $(t_0,x_0)\in[0,T]\times\R^n$ there exists an optimal trajectory $\tilde{x}:[t_0,T]\to\R^n$ in the inverse Bolza problem satisfying $\tilde{x}(t_0)=x_0$.
\end{enumerate}
 \end{Th}

Suppose that the Hamiltonian $H:\R^n\times\R^n\to\R$ satisfies \eqref{A} and it is additionally\linebreak positively homogeneous, i.e. $H(x,\alpha p)=\alpha H(x,p)$ for $\alpha>0$. Then
\begin{equation*}
L(x,v)=\left\{\begin{array}{ccl}
0 & \tn{if} & v\in G(x)\\
+\infty & \tn{if} & v\notin G(x),
\end{array}\right.
\end{equation*}
where the set-valued map $G:\R^n\leadsto \R^n$ is given by
\begin{equation}\label{G}
G(x)=\{\,v\in\R^n\mid \langle v,p\rangle\leq H(x,p)\;\tn{for all}\; p\in\R^n\,\}.
\end{equation}
The set-valued map $G$ is nonempty-compact-convex-valued and Lipschitz continuous with respect to the Hausdorff distance; see \cite[Sect. 7]{HF} or \cite[Sect. 5.3]{F-P-Rz}. The optimal control problem related to the Hamiltonian $H$ being positively homogeneous is called a Mayer problem. The Mayer problem is a special case of the Bolza problem.  It can be formulate in the following way.

\vspace{2mm}
 We consider a differential inclusion
\begin{equation}\label{DI}
\dot{x}(t)\in G(x(t))\;\;\tn{a.e.}\;\,t\in[0,T].
\end{equation}
By $\mathrm{Sol}(G)$ we denote the set of absolutely continuous solutions $x:[0,T]\to\R^n$  of \eqref{DI}. By $\mathrm{Sol}_G(t_0,x_0)$ we denote the set of $x(\cdot)\in \mathrm{Sol}(G)$ that satisfy an initial condition $x(t_0)=x_0$, where $(t_0,x_0)\in[0,T]\times\R^n$. Every solution of the differential inclusion \eqref{DI} is extendable onto the whole interval $[0,T]$. The set $\mathrm{Sol}_G(t_0,x_0)$ is nonempty and compact in the space of continuous functions $C([0,T],\R^n)$ with the supremum norm. Moreover, the set-valued map $\mathrm{Sol}(G):[0,T]\times\R^n\leadsto C([0,T],\R^n)$ is Lipschitz continuous.

\vspace{2mm}
For a given function $g_0:\R^n\to\R$ and an initial pair $(t_0,x_0)$ the Mayer problem consists in finding an optimal trajectory $\bar{x}(\cdot)\in\mathrm{Sol}_G(t_0,x_0)$ such that
\begin{equation*}
g_0(\bar{x}(0))=\inf\{\,g_0(x(0))\mid x\in\mathrm{Sol}_G(t_0,x_0)\,\}.
\end{equation*}
We shall call $g_0$ the initial function in the Mayer problem.

For a given function $g_T:\R^n\to\R$ and an initial pair $(t_0,x_0)$ the inverse Mayer problem consists in finding an optimal trajectory $\tilde{x}(\cdot)\in\mathrm{Sol}_G(t_0,x_0)$ such that
\begin{equation*}
g_T(\tilde{x}(T))=\sup\{\,g_T(x(T))\mid x\in\mathrm{Sol}_G(t_0,x_0)\,\}.
\end{equation*}
We shall call $g_T$ the terminal function in the inverse Mayer problem.

\vspace{2mm}
The value function $U\!:\![0,T]\times\R^n\!\to\!\R$ corresponding to the Mayer problem is given by
\begin{equation}\label{ValueU}
U(t_0,x_0)=\inf\{\,g_0(x(0))\mid x\in\mathrm{Sol}_G(t_0,x_0)\,\}.
\end{equation}
If the initial  function $g_0$ is continuous, then the value function $U$ is continuous  and directly from the definition of  $U$ we obtain that
\begin{align}
&\forall\,x(\cdot)\in\mathrm{Sol}(G),\;\forall\,t_1<t_2,\;\; U(t_1,x(t_1))\geq U(t_2,x(t_2)),\label{U1}\\
&\forall\,(t_0,x_0),\;\exists\,\bar{x}(\cdot)\in\mathrm{Sol}_G(t_0,x_0),\;\forall\,t\in[0,t_0],\;\;U(t,\bar{x}(t))= U(t_0,x_0).\label{U2}
\end{align}

\vspace{2mm}
The value function $W:[0,T]\times\R^n\to\R$ corresponding to the inverse Mayer problem is given by the formula
\begin{equation}\label{ValueW}
W(t_0,x_0)=\sup\{\,g_T(x(T))\mid x\in\mathrm{Sol}_G(t_0,x_0)\,\}.
\end{equation}
If the terminal function $g_T$ is continuous, then the value function $W$ is continuous and
directly from the definition of $W$ we obtain that
\begin{align}
&\forall\,x(\cdot)\in \mathrm{Sol}(G),\;\forall\,t_1<t_2,\;\; W(t_1,x(t_1))\geq W(t_2,x(t_2)),\label{W1}\\
&\forall\,(t_0,x_0),\;\exists\,\tilde{x}(\cdot)\in\mathrm{Sol}_G(t_0,x_0),\;\forall\,t\in[t_0,T],\;\;W(t_0,x_0)= W(t,\tilde{x}(t)).\label{W2}
\end{align}

\section{General setting}\label{s-gs}

\noindent Let $X$, $Y$ be arbitrary sets and $A(X)$, $A(Y)$ be families of extended-real-valued  functions, i.e. $A(X)\subset\{g:X\to\R\cup\{\pm\infty\}\}$ and $A(Y)\subset\{g:Y\to\R\cup\{\pm\infty\}\}$.
We say that an operator $F:A(X)\to A(Y)$ is monotonic if
\begin{equation}\label{monotonic}
\forall\, g,g'\in A(X),\;\;\;g\leq g'\,\Rightarrow\, F(g)\leq F(g').
\end{equation}

\begin{Th}\label{GS}
Assume that the operators $F:A(X)\to A(Y)$ and $B:A(Y)\to A(X)$ are monotonic and  satisfy
\begin{align}
& \forall\, g_0\in A(X),\;\;\; B(F(g_0))\leq g_0,\label{BF}\\
& \forall\, g_1\in A(Y),\;\;\; F(B(g_1))\geq g_1.\label{FB}
\end{align}
Then we have
\begin{align}
& \forall\, g\in A(Y), \;\;\; B(F(B(g)))=B(g),\label{rekBF}\\
& \forall\, g\in A(X), \;\;\; F(B(F(g)))=F(g).\label{rekFB}
\end{align}
\end{Th}

\begin{proof}
By \eqref{BF}, we have $B(F(B(g)))\leq B(g)$. By \eqref{FB}, we have $F(B(g))\geq g$. Since the operator $B$ is monotonic, then $B(F(B(g)))\geq B(g)$, which follows \eqref{rekBF}.\\
By \eqref{FB}, we have $F(B(F(g)))\geq F(g)$. By \eqref{BF}, we have $B(F(g))\leq g$. Since the operator $F$ is monotonic, then $F(B(F(g)))\leq F(g)$, which follows \eqref{rekFB}.
\end{proof}

In the above framework it is reasonable to  say that an initial function $g_0\in A(X)$ is reconstructible if $g_0=B(F(g_0))$ [or equivalently $g_0=B(g_1)$ with $g_1=F(g_0)$]. The class of reconstructible initial condition is given by $B(A(Y))$.
Indeed, if the initial condition is given by $g_0:=B(g)$, for a $g\in A(Y)$, then  $g_0$ is reconstructible.  By \eqref{rekBF} in  Theorem \ref{GS},
\begin{equation}\label{rifgs}
g_0=B(g)=B(F(B(g)))=B(F(g_0)).
\end{equation}
The above scheme is illustrated below:

\begin{center}
\includegraphics[width=12cm,height=5.5cm]{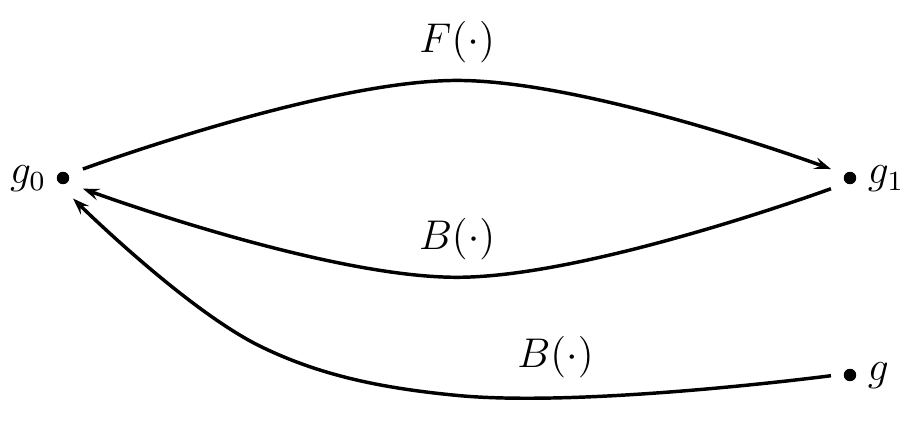}
\end{center} 

The operators $F$ and $B$ will be called forward and backward, respectively.
Below we provide an examples of the forward  and  backward operators  that satisfy the assumptions of Theorem \ref{GS}.

\begin{Ex}[Set-valued Mayer Problem]\label{SV}
Let $X,\,Y$ be arbitrary sets, $A(X),\,A(Y)$ be families of all bounded function. Suppose that a set-valued map $\varphi : X\leadsto Y$ is a surjection, i.e. for every $y\in Y$ there exists $x\in X$ such that $y\in\varphi(x)$. The inverse set-valued map $\varphi^{-1}:Y\leadsto X$ is given by $\varphi^{-1}(y)=\{x\in X\mid\,y\in \varphi(x)\}$. Obviously, $\varphi$ is a surjection if and only if $\varphi^{-1}$ has nonempty values.

The operators $F=F_\varphi:A(X)\to A(Y)$ and $B=B_\varphi:A(Y)\to A(X)$ given by
\begin{align*}
& F_\varphi(g_0)(y):=\inf\{\,g_0(x)\mid x\in\varphi^{-1}(y)\,\}\;\;\tn{for}\;\;g_0\in A(X),\; y\in Y,\\
& B_\varphi(g_1)(x):=\sup\{\,g_1(y)\mid y\in\varphi(x)\,\}\;\;\tn{for}\;\;g_1\in A(Y),\; x\in X
\end{align*}
are monotonic and satisfy \eqref{BF} and \eqref{FB}. Indeed, if $x\in\varphi^{-1}(y)$, then $F(g_0)(y)\leq g_0(x)$. Thus, $B(F(g_0))(x)=\sup\{\,F(g_0)(y)\mid y\in\varphi(x)\,\}\leq g_0(x)$, which follows \eqref{BF}. If $y\in \varphi(x)$, then $B(g_1)(x)\geq g_1(y)$. Therefore, $F(B(g_1))(y)=\inf\{\,B(g_1)(x)\mid y\in\varphi(x)\,\}\geq g_1(y)$, which follows \eqref{FB}.
\end{Ex}

\section{Two proofs of the main result}\label{pmr}
\noindent In the section we provide  the proof of Theorem \ref{main} by two methods. The first methods bases on the classical comparison result for viscosity super- and subsolutions and on the fact that  forward solutions are backward supersolution.
\begin{proof}[Proof of Theorem \ref{main} - viscosity solutions approach.]
We define the forward operator $F:C(\R^n,\R)\to C(\R^n,\R)$ and the backward operator $B:C(\R^n,\R)\to C(\R^n,\R)$  by
\begin{align*}
& F(g_0)=u(T,\cdot),\,\tn{where}\;u\;\tn{is a forward viscosity solution of}\;\eqref{HJ1}\;\tn{satisfying}\; u(0,\cdot)=g_0(\cdot),\\
& B(g_T)=w(0,\cdot),\,\tn{where}\;w\;\tn{is a backward viscosity solution of}\; \eqref{HJ1}\;\tn{satisfying}\; w(T,\cdot)=g_T.
\end{align*}
From (a) and (b) in Theorem \ref{OT1} it follows that the operators $F$ and $B$ are monotonic.

Assume that $v:[0,T]\times \R^n\to\R$ is a backward viscosity solution of \eqref{HJ1}. Moreover, let $u:[0,T]\times \R^n\to\R$ be a forward viscosity solution of \eqref{HJ1} satisfying  $v(0,\cdot)=u(0,\cdot)$ and $w:[0,T]\times \R^n\to\R$ be a backward viscosity solution of \eqref{HJ1} satisfying $w(T,\cdot)=u(T,\cdot)$.

By Theorem \ref{OT1} (c), $u$ is a bilateral viscosity supersolution of \eqref{HJ1}, in particular $u$ is also a backward viscosity supersolution of \eqref{HJ1}.  Since
$w$ and $u$ are, respectively, a backward viscosity sub- and supersolution of \eqref{HJ1} and $w(T,\cdot)=u(T,\cdot)$, by  Theorem \ref{OT1} (b), we have the inequality $w\leq u$. It follows \eqref{BF} in Theorem \ref{GS}.

 By Theorem \ref{OT1} (c), the backward viscosity solution  $v$ is also a bilateral viscosity subsolution, in particular $v$ is also a forward viscosity subsolution of \eqref{HJ1}. Since  $v$ and $u$ are, respectively, a forward viscosity sub- and supersolution of \eqref{HJ1} and $v(0,\cdot)= u(0,\cdot)$, by  Theorem \ref{OT1} (a), we have the inequality $v\leq u$. It follows \eqref{FB} in Theorem \ref{GS}.

Thus, the assumptions of Theorem \ref{GS} are satisfied for $F$ and $B$.  By  \eqref{rifgs}, we have $g_0=B(F(g_0))$ for $g_0\!=\!B(g)$ and $g(\cdot)\!=\!v(T,\cdot)$.  It means that $g_0$ is reconstructible in time $T$.

Since $v(T,\cdot)\leq u(T,\cdot)=w(T,\cdot)$ and $v$ and $w$ are, respectively, a backward viscosity sub- and supersolution of \eqref{HJ1}, by  Theorem \ref{OT1} (b), we obtain the inequality $v\leq w$.
\end{proof}

\begin{proof}[Proof of Theorem \ref{main} - optimal control approach.]
By Theorem \ref{OT}, the solutions $v$, $u$, $w$ can be represented as value functions $V$, $U$, $W$ of a corresponding optimal control problem. Let $V=v$ be the value function in the inverse Bolza problem with the terminal function $g(\cdot)=v(T,\cdot)$. Let $U=u$ be the value function in the Bolza problem with the initial function $g_0(\cdot)=V(0,\cdot)$. Let $W=w$ be the value function in the inverse Bolza problem with the terminal function $g_T(\cdot)=U(T,\cdot)$.

Let $t\in[0,T]$. We define monotonic operators $F_t,B_t:C(\R^n,\R)\to C(\R^n,\R)$ by
\begin{align*}
& F_t(g_0)(y_0)= \inf_{\begin{array}{c}
\scriptstyle x(\cdot)\,\in\,\mathcal{A}([0,T],\R^n)\\[-1mm]
\scriptstyle x(t)=y_0
\end{array}}\,\Big\{\,g_0(x(0))+\int_{0}^{t}L(x(s),\dot{x}(s))\,\it{ds}\,\Big\},\\
& B_t(g_T)(x_0)= \sup_{\begin{array}{c}
\scriptstyle x(\cdot)\,\in\,\mathcal{A}([0,T],\R^n)\\[-1mm]
\scriptstyle x(T-t)=x_0
\end{array}}\,\Big\{\,g_T(x(T))-\int_{T-t}^{T}L(x(s),\dot{x}(s))\,\it{ds}\,\Big\}.
\end{align*}
In view of the Change of Variables Theorem we have
\begin{equation*}
B_t(g_T)(x_0)= \sup_{\begin{array}{c}
\scriptstyle x(\cdot)\,\in\,\mathcal{A}([0,T],\R^n)\\[-1mm]
\scriptstyle x(0)=x_0
\end{array}}\,\Big\{\,g_T(x(t))-\int_{0}^{t}L(x(s),\dot{x}(s))\,\it{ds}\,\Big\}.
\end{equation*}
To show that $B_t(F_t(g_0))\!\leq\! g_0$ we fix $x_0\!\in\!\R^n$. For arbitrary $x(\cdot)\!\in\!\mathcal{A}([0,T],\R^n)$ with $x(0)=x_0$
\begin{equation*}
F_t(g_0)(x(t))\;\;\leq\;\; g_0(x_0)+\int_{0}^{t}L(x(s),\dot{x}(s))\,\it{ds}.
\end{equation*}

\noindent Thus
\begin{equation*}
\sup_{\begin{array}{c}
\scriptstyle x(\cdot)\,\in\,\mathcal{A}([0,T],\R^n)\\[-1mm]
\scriptstyle x(0)=x_0
\end{array}}\,\Big\{\,F_t(g_0)(x(t))-\int_{0}^{t}L(x(s),\dot{x}(s))\,\it{ds}\,\Big\}\;\;\leq\;\; g_0(x_0).
\end{equation*}
So, the assumption \eqref{BF} is satisfied for $F=F_t$, $B=B_t$.

\noindent To show that $F_t(B_t(g_T))\!\geq\! g_T$ we fix $y_0\in\R^n$. For arbitrary $x(\cdot)\!\in\!\mathcal{A}([0,T],\R^n)$ with $x(t)=y_0$
\begin{equation*}
B_t(g_T)(x(0))\;\;\geq\;\;g_T(y_0)-\int_0^t L(x(s),\dot{x}(s))ds.
\end{equation*}
Thus
\begin{equation*}
\inf_{\begin{array}{c}
\scriptstyle x(\cdot)\,\in\,\mathcal{A}([0,T],\R^n)\\[-1mm]
\scriptstyle x(t)=y_0
\end{array}}\,\Big\{\,B_t(g_T)(x(0))+\int_{0}^{t}L(x(s),\dot{x}(s))\,\it{ds}\,\Big\}\;\;\geq\;\;g_T(y_0).
\end{equation*}
So, the assumption \eqref{FB} is satisfied for $F=F_t$, $B=B_t$.

We have $B_t(g)=V(T-t,\cdot)$, $F_t(g_0)=U(t,\cdot)$ and $B_t(g_T)=W(T-t,\cdot)$.
The assumptions of Theorem \ref{GS} are satisfied for $F=F_T$ and $B=B_T$.  By  \eqref{rifgs}, we have $g_0=B_T(F_T(g_0))$ for $g_0=B_T(g)$.  It means that $g_0$ is reconstructible in time $T$.

To show that $V\leq W$ observe that, by \eqref{FB}, we have $g_T=F_T(B_T(g))\geq g$. So we get
$$V(t_0,x_0)=B_{T-t_0}(g)(x_0)\leq B_{T-t_0}(g_T)(x_0)=W(t_0,x_0).$$
By the Dynamic Programming Principle for the Bolza problem, we have $F_{t_1+t_2}=F_{t_1}\circ F_{t_2}$ for $0\leq t_2\leq t_1+t_2\leq T$. Therefore we get
$$W(t,\cdot)=B_{T-t}(g_T)=B_{T-t}(F_T(g_0))=B_{T-t}(F_{T-t}(F_t(g_0)))=B_{T-t}(F_{T-t}(U(t,\cdot))).$$
By \eqref{BF}, we have $B_{T-t}(F_{T-t}(U(t,\cdot)))\leq U(t,\cdot)$. Therefore $W\leq U$.
\end{proof}

\section{Does the forward solution equal to the backward one?}

\noindent In Theorem \ref{main} we describe the class of reconstructible initial condition. To a reconstructible in time $T>0$ initial function $g_0$  it corresponds a forward viscosity solutions $u$ of \eqref{HJ1} satisfying $u(0,\cdot)=g_0(\cdot)$ and a backward viscosity solution $w$ of \eqref{HJ1} satisfying $w(T,\cdot)=u(T,\cdot)$. In the section we ask whether the equality $u=w$ holds true on the whole domain $[0,T]\times\R^n$. From the definition it holds true on the set $\{0,T\}\times\R^n$. The problem can be reformulated to the language of value functions. In the section we consider three particular cases when we are able to answer the question asked in the title of the section.

\vspace{2mm}
 In Subsection \ref{s-dmp} we consider a Mayer type control problem with a finite state space $X$ and a discrete time dynamics given be a set-valued map $\phi:X\leadsto X$. We adopt our general scheme to the case and describe the set of reconstructible initial functions. We provide an example of a reconstructible initial function for which the value function $U$ in the  Mayer problem is \bf{not} equal to the value function $W$ in the inverse Mayer problem.  This surprising example motivated us to present in the paper the discrete time model.

\vspace{2mm}
In Subsection \ref{s-cmp} we show that in one dimensional case ($n=1$) and for Hamiltonians corresponding to a Mayer problem the forward viscosity solution corresponding to a recontructible initial function is a bilateral solution, i.e. $u=w$.

\vspace{2mm}
In Subsection \ref{s-cbp}, we consider concave-convex Hamiltonians. We show that forward viscosity solutions corresponding to a convex initial functions are bilateral ones. It means that for such Hamiltonians every convex initial function is reconstructible in every time $T>0$. In the proof we use Pontriagin Maximum Principle in a version obtained in\cite{R-W-1}.

\subsection{Discrete time Mayer Problem}\label{s-dmp}
 We suppose that the state space is finite $X=\{1,\,2,\,\ldots, n\}$ and a set-valued map $\phi:X\leadsto X$ is a surjection. Fix $T\in\N$ and denote $I=\{0,\,1\,\ldots,\,T\}$. We say that $x:I\to X$ is a trajectory if $x(i+1)\in\phi(x(i))$ for $i=0,\,1,\ldots, T-1$. We denote by $\mathrm{Sol}(\phi)$ the set of all trajectories. Let $\mathrm{S}_\phi(k,x_k)=\{x\in\mathrm{Sol}(\phi)\mid x(k)=x_k\}$, where $(k,x_k)\in I\times X$. Fix functions $g_0,\,g_T:X\to\R$.  The  value function $U:I\times X\to \R$ in the discrete time Mayer problem and the value function $W:I\times X\to\R$ in the inverse Mayer problem are given by
\begin{align}
&U(k,x_k)=\min\{\,g_0(x(0))\,\mid x\in\mathrm{S}_\phi(k,x_k)\,\},\label{Value U d}\\
&W(k,x_k)=\max\{\,g_T(x(T))\,\mid x\in\mathrm{S}_\phi(k,x_k)\,\}.\label{Value W d}
\end{align}
The functions $U$, $W$ are forward non-increasing along any trajectory (comp. \eqref{U1}, \eqref{W1}). The function $U$ is backward in time constant along the optimal trajectory and the function $W$ is forward in time constant along the optimal trajectory (comp. \eqref{U2}, \eqref{W2}).

\vspace{2mm}
A function $g:X\to \R$ can be treated as a sequence $(g(1),g(2),\ldots,g(n))\in\R^n$. Let $G=\{\,g:X\to\R\,\}$. We define a forward one step operator $F_1:G\to G$ and a backward one step operator $B_1:G\to G$  by
\begin{align*}
&F_1(g)(j)=\min\{\,g(i)\,\mid\,j\in\phi(i)\,\},\\
&B_1(g)(i)=\max\{\,g(j)\,\mid\,j\in\phi(i)\,\}.
\end{align*}

\vspace{2mm}
Suppose that $U$ and $W$ be given by \eqref{Value U d} and \eqref{Value W d}, respectively. In view of the Dynamic Programming Principle we obtain that
\begin{align}
&U(k,\cdot)=F_1^k(g_0)\;\;\tn{for}\;\; k=1,\ldots,T,\label{DPP U}\\
&W(k,\cdot)=B_1^{T-k}(g_T)\;\;\tn{for}\;\; k=0,\ldots,T-1.\label{DPP W}
\end{align}

\vspace{2mm}
The reconstruction result for discrete dynamics can be formulated in the following way.

\begin{Prop}\label{reconstruction d}
Let $g\in G$ and $g_0=B_1^T(g)$. Suppose that $U:I\times X\to\R$ is the forward value function given by \eqref{Value U d}. Set $g_T=U(T,\cdot)$. Suppose that  $W:I\times X\to\R$ is the backward value function given by \eqref{Value W d}. Then $U(0,\cdot)=W(0,\cdot)$.
\end{Prop}

\begin{proof}
We define the forward operator $F=F_1^T$ and the backward operator $B=B_1^T$. The set-valued map $\varphi:X\leadsto X$ is given by $\varphi(x_0)=\phi^T(x_0)=\{\,x_T\mid x\in\mathrm{Sol(}\phi),\;x(0)=x_0\,\}$. We observe that $F=F_\varphi$ and $B=B_\varphi$, where $F_\varphi,\,B_\varphi$ are given in Example \ref{SV}. So, the operators $F,\,B$ satisfy the assumptions of Theorem \ref{GS}. Since $g_0=B_1^T(g)$ and $g_T=U(T,\cdot)=F_1^T(g_0)$, we deduce from \eqref{rifgs} that $U(0,\cdot)=g_0=B_1^T(g_T)=W(0,\cdot)$.
\end{proof}

\vspace{-6mm}
\pagebreak
Below we provide an example of a discrete Mayer problem  such that $U\neq W$.

\begin{Ex}\label{exbr}
Let $X=\{1,2,3\}$. To describe a set-valued map $\phi:X\leadsto X$ we use a zero-one matrix $A(3\times3)$ such that $j\in\phi(i)\Leftrightarrow a_{ij}=1$.
\begin{equation*}
\left[\begin{array}{ccc}
     0  &   1  &   0  \\
     1  &   0  &   1  \\
     1  &   0  &   0
\end{array}\right]
\end{equation*}
Thus $\phi(1)=\{2\},\,\;\phi(2)=\{1,\;3\},\;\phi(3)=\{1\}$.
We set $T=3$, $I=\{0,\;1,\;2,\;3\}$ and $g=(1,\,   2,\,     3)$.
The backward value function $V:I\times X\to\R$ corresponding to the terminal function $g$ is given by
\begin{equation*}
\begin{array}{lccc}
V(3,\cdot)=&
     (\;1 & 2 & 3\;)\\
V(2,\cdot)=&
     (\;2& 3 & 1\;)\\
V(1,\cdot)=&
     (\;3 & 2 & 2\;)\\
V(0,\cdot)=&
    (\;2 & 3 & 3\;)
\end{array}
\end{equation*}
The function $g_0=(2,\;3,\;3)$ is reconstructible in time $T=3$. The forward value function $U:I\times X\to\R$ corresponding to the initial function $g_0$ is given by
\begin{equation*}
\begin{array}{lccc}
U(0,\cdot)=&
     (\;2 & 3 & 3\;)\\
U(1,\cdot)=&
     (\;3 & 2 & 3\;)\\
U(2,\cdot)=&
     (\;2 & 3 & 2\;)\\
U(3,\cdot)=&
    (\;2 & 2 & 3\;)
\end{array}
\end{equation*}
We set $g_3=(2,\;2,\;3)$. The backward value function $W$ corresponding to the terminal function $g_3$ is given by
\begin{equation*}
\begin{array}{lccc}
W(3,\cdot)=&
     (\;2 & 2 & 3\;)\\
W(2,\cdot)=&
     (\;2 & 3 & 2\;)\\
W(1,\cdot)=&
     (\;3 & 2 & 2\;)\\
W(0,\cdot)=&
    (\;2 & 3 & 3\;)
\end{array}
\end{equation*}
Obviously, we have  $U(0,\cdot)=W(0,\cdot)$ and $U(3,\cdot)=W(3,\cdot)$. Nevertheless, $U(1,\;3)\neq W(1,\;3)$.
\end{Ex}

\subsection{Continuous time Mayer Problem}\label{s-cmp}
In the section we show that a forward viscosity solution corresponding to a reconstructible initial function is a bilateral viscosity solution for Hamiltonians that are positively homogeneous with respect to the second variable in dimension $n=1$. It means that for a continuous time Mayer problem in one dimensional state space a situation described in Example \ref{exbr} is not possible. 

\vspace{2mm}
We consider the Mayer control problem given by a differential inclusion \eqref{DI}, where the right hand side $G$ is given by \eqref{G} and the Hamiltonian $H$ satisfies \eqref{A} and is positively homogeneous with respect to the second variable. Alternatively we can start formulating the problem from a set-valued map $G:\R^n\leadsto \R^n$. If $G$ is a Lipschitz continuous map with nonempty, compact, convex values, then the Hamiltonian $H:\R^n\times\R^n\to\R$ given by
$$H_G(x,p)=\sup\{\,\langle v,p\rangle\mid v\in G(x)\,\}$$
satisfies \eqref{A} and is positively homogeneous with respect to the second variable.

\pagebreak
Define the set-valued map $\varphi_G:\R^n\leadsto\R^n$ by the formula
$$\varphi_G(x_0)=\{\,x(T)\mid x\in\mathrm{Sol}_G(0,x_0)\,\}.$$
The set-valued map $\varphi_G$ is a Lipschitz  continuous  surjection with nonempty and compact values. The inverse set-valued map $\varphi_G^{-1}$ is given by
$$\varphi_G^{-1}(x_0)=\{\,x(0)\mid x\in\mathrm{Sol}_G(T,\,x_0)\,\}.$$
So, $\varphi_G^{-1}$ is a Lipschitz continuous map with nonempty and compact values.

\vspace{2mm}
We set  in the scheme presented in Section \ref{s-gs}: $X=Y=\R^n$ and $A(\R^n)$ is the space of continuous functions $C(\R^n,\,\R)$.  The operators $F_{\varphi_G}$, $B_{\varphi_G}$ defined in the Example \ref{SV} transform continuous functions into continuous functions. By the same arguments as in Example \ref{SV} the operators  $F_{\varphi_G}$, $B_{\varphi_G}$ satisfy the assumptions of Theorem \ref{GS}.

\vspace{2mm}
Fix $g\in C(\R^n,\R)$. Let  $V$ be the value function in the  inverse Mayer problem with terminal function $g$. Then we obtain that $V(0,\cdot)=B_{\varphi_G}(g)$. Let $U$ be the value function corresponding to the Mayer problem with  an initial function $g_0(\cdot):=V(0,\cdot)$. Then we obtain that $U(T,\cdot)=F_{\varphi_G}(g_0)$.  Let $W$ be the value function corresponding to the inverse Mayer problem and $W(T,\cdot)=U(T,\cdot)$. If $g_T:=W(T,\cdot)=U(T,\cdot)$, then  by  \eqref{rifgs} we have $g_0=B_{\varphi_G}(g_T)=W(0,\cdot)$.

\vspace{2mm}
As a corollary of Theorem \ref{GS} and the above consideration we obtain:

\begin{Cor}
Assume that $G:\R^n\leadsto\R^n$ is a Lipschitz continuous map with nonempty, convex, compact values. Then an initial function $g_0\in C(\R^n,\,\R)$ is reconstructible if and only if there exists  a  function $g\in C(\R^n,\,\R)$ such that $g_0=V(0,\cdot)$ and $V$ is the value function corresponding to the inverse Mayer problem with the terminal function $g$.
\end{Cor}

It appears a natural question whether $U=W$ on $(0,T)\times\R^n$ if $U=W$ on $\{0,T\}\times \R^n$.  This question correlates well with Proposition 4.4 in \cite{C-F}. We succeed in proving that $U=W$ only in a very special case: $n=1$ and  $H$ corresponding to a Mayer problem.

\begin{Th}\label{UV}
Assume that $G:\R^1\leadsto\R^1$ is a Lipschitz continuous map with nonempty, convex, closed values and functions $U,\,W:[0,T]\times\R^1\to\R$ are continuous and satisfy \eqref{U1}-\eqref{U2} and \eqref{W1}-\eqref{W2}, respectively. If $U(0,\cdot)=W(0,\cdot)$ and $U(T,\cdot)=W(T,\cdot)$, then\linebreak we obtain $U=W$.
\end{Th}

\begin{proof}
First, we show that  $W\leq U$. Fix $(t_0,x_0)$. By \eqref{W2}, we choose $\tilde{x}(\cdot)\in\mathrm{Sol}_G(t_0,x_0)$. The latter, together with \eqref{U1}, implies that $$W(t_0,x_0)\leq W(T,\tilde{x}(T))=U(T,\tilde{x}(T))\leq U(t_0,x_0).$$

Next, suppose that $U\neq W$. Then there exists $(t_0,x_0)\in(0,T)\times\R$ such that
$$u_0:=U(t_0,x_0)>w_0:= W(t_0,x_0)$$
By \eqref{U2} and \eqref{W2}, we choose $\bar{x}(\cdot),\,\tilde{x}(\cdot)\in\mathrm{Sol}_G(t_0,x_0)$ such that
\begin{align*}
& U(t,\bar{x}(t))=u_0\hspace{3mm}\tn{for all}\;\;t\in[0,t_0],\\
& W(t,\bar{x}(t))=w_0\;\;\tn{for all}\;\;t\in[t_0,T].
\end{align*}

\vspace{-4mm}
\pagebreak
\noindent By \eqref{W2}, we choose $\tilde{z}(\cdot)\in\mathrm{Sol}_G(0,\bar{x}(0))$. So
$$W(t,\tilde{z}(t))=W(0,\tilde{z}(0))=W(0,\bar{x}(0))=U(0,\bar{x}(0))=u_0\;\;\tn{for all}\;\;t\in[0,T].$$
By \eqref{U2} we choose $\bar{z}(\cdot)\in\mathrm{Sol}_G(T,\tilde{x}(T))$. So
$$U(t,\bar{z}(t))=U(T,\bar{z}(T))=U(T,\tilde{x}(T))=W(T,\tilde{x}(T))=w_0\;\;\tn{for all}\;\;t\in[0,T].$$
Since $\tilde{z}(\cdot)$, $\bar{z}(\cdot)$ are optimal trajectories then the functions $U$, $W$ are constant along them. Consider the case $\bar{z}(0)<\tilde{z}(0)$. Since $U(0,\cdot)$ is continuous then there exists $z_0\in(\bar{z}(0),\tilde{z}(0))$ such that $U(0,z_0)=\frac{u_0+w_0}{2}$. By \eqref{W2} we choose optimal trajectory $\hat{z}:[0,T]\to\R$ for the initial condition $(0,z_0)$. Along the trajectory $(t,\hat{z}(t))$ the functions $U,\,W$ are constant and equal to $\frac{u_0+w_0}{2}$. The trajectory $\hat{z}(\cdot)$ does not cross trajectories $\tilde{z}(\cdot)$ and $\bar{z}(\cdot)$. Therefore $\hat{z}(T)\in(\bar{z}(T),\,\tilde{z}(T))$. Thus, the trajectory $\hat{z}(\cdot)$ has to cross the trajectory $\hat{x}(\cdot)$, where
\begin{equation}
\hat{x}(t)=\left\{ \begin{array}{lcc}
\bar{x}(t) &\!\!\tn{if}\!\!\!\!& t\in[0,t_0],\\
\tilde{x}(t) &\!\!\tn{if}\!\!\!& t\in[t_0,T].
\end{array}\right.
\end{equation}
If exists $t_1\in[0,t_0]$ such that $\hat{z}(t_1)=\bar{x}(t_1)$, then $U(t_1,\bar{x}(t_1))=u_0$ and $U(t_1,\hat{z}(t_1))=\frac{u_0+w_0}{2}$, which follows a contradiction.
If there exists $t_2\in[t_0,T]$ such that $\tilde{x}(t_2)=\hat{z}(t_2)$, then $W(t_2,\tilde{x}(t_2))=w_0$ and $ W(t_2,\hat{z}(t_2))=\frac{u_0+w_0}{2}$, which follows the contradiction.
\end{proof}

Theorem \ref{UV} can be reformulated in the following way.

\begin{Cor}
Assume that the Hamiltonian $H:\R\times\R\to\R$ satisfies \eqref{A} and $H(x,\cdot)$ is positively homogeneous. If an initial function $g_0\in C(\R,\R)$ is reconstructible in time $T$ and $u:[0,T]\times\R\to\R$ is the forward viscosity solution of \eqref{HJ1} satisfying $u(0,\cdot)=g_0(\cdot)$, then $u$ is a bilateral viscosity solution of \eqref{HJ1}.
\end{Cor}

\subsection{Bolza Problem with a convex initial function}\label{s-cbp}
We show that if the Hamiltonian $H(x,p)$ is concave with respect to the first variable and is convex with respect to the second variable then every convex initial function $g_0$ is reconstructible in every positive time $T$. Moreover, equality $u=v$ holds true in this case.

\begin{Prop}\label{max}
Assume that $H:\R^n\times\R^n\to\R$ satisfies \eqref{A} and that $g_0,\,g_T\in C(\R^n,\R)$. Let $L:\R^n\times\R^n\to\R\cup\{+\infty\}$ be given by \eqref{tran1} and $U,W:[0,T]\times\R^n\to\R$ be given by \eqref{fvf1}, \eqref{fvf2}, respectively. If $\bar{x}(\cdot)\in\mathcal{A}([0,T],\R^n)$ is a maximal optimal trajectory in the Bolza problem and $U(T,\cdot)=W(T,\cdot)$, then $\bar{x}(\cdot)$ is a maximal optimal trajectory in the inverse Bolza problem and $U(t,\bar{x}(t))=W(t,\bar{x}(t))$ for all $t\in[0,T]$.
 \end{Prop}

\begin{proof}
If $U(T,\cdot)=W(T,\cdot)$ and $\bar{x}:[0,T]\to\R^n$ is a maximal optimal trajectory in the Bolza problem, then by the Dynamic Programming Principle for the Bolza problem we obtain
\begin{eqnarray*}
U(t,\bar{x}(t))  &=& U(T,\bar{x}(T))-\int_t^T L(\bar{x}(s),\dot{\bar{x}}(s))\,\it{ds}\\
&=& W(T,\bar{x}(T))-\int_t^T L(\bar{x}(s),\dot{\bar{x}}(s))\,\it{ds}\;\;\leq\;\; W(t,\bar{x}(t))
\end{eqnarray*}
for every $t\in[0,T]$. By Theorem \ref{main} we have that $W\leq U$. 

\pagebreak
\noindent Therefore, for every $t\in[0,T]$,
\begin{equation*}
U(t,\bar{x}(t))=W(t,\bar{x}(t))=g_T(\bar{x}(T))-\int_t^T L(\bar{x}(s),\dot{\bar{x}}(s))\,\it{ds}.
\end{equation*}
Taking $t=0$ in the above property  we conclude that $\bar{x}(\cdot)$ is a maximal optimal trajectory in the inverse Bolza problem.
\end{proof}

\begin{Th}\label{vbshjb}
Assume that the Hamiltonian $H:\R^n\times\R^n\to \R$ satisfies
\begin{equation}\label{B}
\left\{\begin{array}{l}
p\to H(x,p)\;\; \it{is convex for every}\;\;  x\in\R^n,\\[0.3mm]
x\to H(x,p)\;\; \it{is concave for every}\;\; p\in\R^n,\\[0.3mm]
|H(x,p)|\leq M(1+|x|)(1+|p|)\;\; \it{for all}\;\; x,p\in\R^n\;\; \it{and some}\;\; M\geq 0.
\end{array}\right.\tag{B}
\end{equation}
If an initial function $g_0:\R^n\to\R$ is convex, then the unique forward viscosity solution $u:[0,\infty)\times\R^n\to\R$ of \eqref{HJ1} satisfying an initial condition $u(0,\cdot)=g_0(\cdot)$ is a bilateral viscosity solution. It means that every convex initial function is reconstructible in every time $T>0$.
\end{Th}

By \cite[Lemma 4.1]{AM3} we obtain that if a Hamiltonian $H:\R^n\times\R^n\to\R$ satisfies \eqref{B}, then it  satisfies \eqref{A}. For example, the Hamiltonian $H:\R^n\times\R^n\to \R$ given by the formula
\begin{equation*}
H(x,p)=\phi(\,|\,A\cdot p\,|\,)-\psi(\,|\,B\cdot x\,|\,)+\langle\,C\cdot x,\,D\cdot p\,\rangle,
\end{equation*}
where   $\phi,\psi:[0,\infty)\to\R$ are nondecreasing, convex functions with linear growth and $A$, $B$, $C$, $D$ are $n\times n$ real matrices, satisfies \eqref{B}. However, the Hamiltonian $H:\R^n\times\R^n\to \R$ defined by $H(x,p)=|x|\,|p|$ satisfies \eqref{A}, but it does not satisfy \eqref{B}.

\vspace{2mm}
In the proof of Theorem \ref{vbshjb} we need the following result of Rockafellar-Wolenski \cite{R-W-1}.
\begin{Th}\label{Th_R_W}
Assume that the Hamiltonian $H:\R^n\times\R^n\to \R$ satisfies \eqref{B} and the initial function $g_0:\R^n\to\R$ is convex. Then for every optimal trajectory $x(\cdot)\,\in\,\mathcal{A}([0,t_0],\R^n)$  in the Bolza problem there exists $p(\cdot)\,\in\,\mathcal{A}([0,t_0],\R^n)$ such that
\begin{equation}\label{zmp}\begin{split}
&\dot{x}(t)\in\partial_{-}^{\,p}H(x(t),p(t)),\;\; -\dot{p}(t)\in\partial_{+}^{\,x}H(x(t),p(t)),\;\;\tn{a.e.}\;\;t\in[0,t_0], \\[1mm]
&p(0)\in\partial_{-}g_0(x(0))\;\;\tn{and}\;\;x(t_0)=x_0.
\end{split}\end{equation}
Conversely, if  functions $x(\cdot),p(\cdot)\,\in\,\mathcal{A}([0,t_0],\R^n)$ solve system \eqref{zmp}, then $x(\cdot)$ is an optimal trajectory in the Bolza problem on the time interval $[0,\,t_0]$.
\end{Th}
The subdifferential $\partial_{-}^{\,p}H(x,y)$  of the function $H(x,\cdot)$ at the point $y$  and the superdifferential $\partial_{+}^{\,x}H(x,y)$ of the function $H(\cdot,y)$ at the point $x$ was recalled in Definition \ref{vs}.  The subdifferential of a convex function and the superdifferential of a concave function can be defined in a simpler way (comp. \cite{R-W}).
\begin{proof}[Proof of Theorem \ref{vbshjb}]
We define the set-valued map $E:\R^{2n}\leadsto\R^{2n}$ by the formula
\begin{equation*}
E(x,p)\,=\,\partial_{-}^{\,p}H(x,p)\times-\partial_{+}^{\,x}H(x,p).
\end{equation*}
In view of \cite[Sect. 6]{R-W-1}, the set-valued map $E$ has nonempty, compact, convex values and closed graph. By the assumption \eqref{A}, we have $\|\partial_{-}^{\,p}H(x,p)\|\leq M(1+|x|)$ and $\|\partial_{+}^{\,x}H(x,p)\|\leq M(1+|p|)$ for all $x,p\in\R^n$, where $\|K\|=\sup_{\xi\in K}|\xi|$. Therefore, $\|E(x,p)\|\leq 2M(1+|(x,p)|)$  for all $x,p\in\R^n$.

Fix $(t_0,x_0)\in(0,T)\times\R^n$. Then there exists an optimal trajectory $x(\cdot)\,\in\,\mathcal{A}([0,t_0],\R^n)$  in the Bolza problem such that $x(t_0)=x_0$. By Theorem \ref{Th_R_W}, there exists $p(\cdot)\,\in\,\mathcal{A}([0,t_0],\R^n)$ solving system \eqref{zmp}. Thus, the pair $(x,p)(\cdot)$ is the solution of the differential inclusion
\begin{equation*}
(\dot{x}(t),\dot{p}(t))\in E(x(t),p(t))\;\;\tn{a.e.}\;\,t\in[0,t_0].
\end{equation*}
Through the above properties of $E$, the theory of differential inclusions \cite{A-C} ensures the possibility of the extension of the solution $(x,p)(\cdot)$ from $[0,t_0]$ to $[0,T]$.  By Theorem~\ref{Th_R_W}, the condition \ref{zmp} is sufficient to optimality of $x(\cdot)$. So, the function $x(\cdot)$ extended onto $[0,T]$  is the maximal optimal trajectory in the Bolza problem. By Proposition \ref{max}, we obtain that $U(t,x(t))=W(t,x(t))$ for all $t\in[0,T]$. In particular, $U(t_0,x_0)=W(t_0,x_0)$. Since $(t_0,x_0)\in(0,T)\times\R^n$ was arbitrary, we conclude that $U=W$ on $(0,T)\times\R^n$. By continuity of $U$ and $W$ we have  $U=W$ on $[0,T]\times\R^n$.
\end{proof}

\begin{Rem}
In \cite{B-CD}, it was shown that a semiconvex functions that are forward viscosity solutions are bilateral ones. Moreover, in Theorem 7.4.13 from  \cite{C-S-2004} it was shown that the value function in the optimal control Bolza problem is semiconvex. We do not know whether  for a Hamiltonian satisfying \eqref{B} does exist an optimal control representation satisfying the assumptions of \cite[Thm. 7.4.13]{C-S-2004}. Therefore, based on Theorem 7.4.13 from  \cite{C-S-2004}, we cannot conclude that the value function corresponding to the variational Bolza problem is semiconvex.
\end{Rem}



\end{document}